\documentclass[12pt]{amsart} 

\usepackage{amsmath, amssymb, amsthm}

\usepackage{fullpage} 
\usepackage{braket}
\usepackage{eucal}
\usepackage[mathletters]{ucs}
\usepackage[utf8x]{inputenc}

\usepackage{tikz-cd}

\DeclareMathOperator{\sSet}{\mathcal{S}et_{\Delta}} 
\DeclareMathOperator{\cA}{\mathcal{A}}  
\DeclareMathOperator{\sk}{\mathrm{sk}}

\theoremstyle{plain}
\newtheorem{thm}{Theorem}[section] % first theorem in section 1 will be 1.1

\newtheorem{theorem}{Theorem}[section]
\newtheorem{lemma}[theorem]{Lemma}
\newtheorem{proposition}[theorem]{Proposition}
\newtheorem{corollary}[theorem]{Corollary}

\begin{document} 

\title[Stability for inner fibrations revisited]{Stability for inner fibrations revisited}
  \author[D.\ Stevenson]{Danny Stevenson}
  \address[Danny Stevenson]
  {School of Mathematical Sciences\\
  University of Adelaide\\
  Adelaide, SA 5005 \\
  Australia}
  \email{daniel.stevenson@adelaide.edu.au}


\begin{abstract}
In this paper we prove a stability result for inner fibrations in terms 
of the wide, or fat join operation on simplicial sets.  We also prove 
some additional results on inner anodyne morphisms that may be of independent interest.    
\end{abstract}
\maketitle

\section{Introduction} 
\label{sec:intro}

Recall (see \cite{EP,J1,HTT,RV}) that the {\em join} of simplicial sets $A$ and $B$ is the 
simplicial set $A\star B$ whose set of $n$-simplices for $n\geq 0$ 
is given by 
\[
(A\star B)_n = A_n \sqcup B_n \sqcup \bigsqcup_{i+j=n-1} A_i\times B_j. 
\]
For a fixed simplicial set $A$, the functor $A\star (-)\colon \sSet\to (\sSet)_{A/}$ 
commutes with colimits and hence 
admits a right adjoint.  If $p\colon A\to X$ is an object in $(\sSet)_{A/}$ then 
its image under this right adjoint is the simplicial set 
$X_{p/}$ defined in Section 1.2.9 of \cite{HTT}.  A fundamental result 
in the theory of $\infty$-categories is the following stability 
result due to Joyal (see Theorem 3.8 of \cite{J1}) and Lurie (see Proposition 2.1.2.1 of \cite{HTT}).  

\begin{thm}[\cite{J1,HTT}] 
\label{thm:A}
Suppose given a diagram 
\[
L\xrightarrow{u} X \xrightarrow{p} S 
\]
of simplicial sets and maps between them.  Suppose that 
$p\colon X\to S$ is an inner fibration between simplicial sets. 
If $K\subset L$ is a subcomplex 
and $v= u|_K$ then the induced map 
\[
X_{u/}\to X_{v/}\times_{S_{pv/}}S_{pu/} 
\]
is a left fibration which is a trivial Kan fibration if the inclusion 
$K\subset L$ is inner anodyne.  
\end{thm} 

There is another join construction for simplicial sets, due 
originally to Joyal and developed further in \cite{HTT} (this notion of join  
also plays a prominent role in \cite{RV}).  If $A$ and $B$ are simplicial 
sets then the {\em fat} or {\em wide join} is the simplicial set $A\diamond B$ defined 
by the pushout diagram 
\[
\begin{tikzcd} 
 \partial\Delta^1 \times A\times B\arrow[d] \arrow[r] & \Delta^1\times A\times  B \arrow[d] \\ 
A\sqcup B \arrow[r] & A\diamond B
\end{tikzcd}
\]
% \[
% \xymatrix{ 
%  \partial\Delta^1 \times A\times B\ar[d] \ar[r] & \Delta^1\times A\times  B \ar[d] \\ 
% A\sqcup B \ar[r] & A\diamond B}
% \]
in $\sSet$.  It follows easily that the set $(A\diamond B)_n$ of 
$n$-simplices of $A\diamond B$ is given by 
\[
(A\diamond B)_n = A_n \sqcup \bigsqcup^{n-1}_{i=1} (A_n\times B_n) \sqcup B_n.  
\]
The wide join has the advantage over the join construction 
described above in that it is simple to define, and is a fairly 
robust construction --- it is a homotopy pushout in the Joyal 
model structure on simplicial sets.  There is also a 
canonical comparison map $A\diamond B\to A\star B$ for any simplicial 
sets $A$ and $B$, this turns out to be a categorical equivalence 
(see Proposition 4.2.1.2 of \cite{HTT}).  In other respects, the 
wide join is not as convenient as the join construction; for 
instance, there is an isomorphism $\Delta^0\diamond \Delta^0\simeq \Delta^1$, but 
there is not an isomorphism $\Delta^0\diamond \Delta^1\simeq \Delta^2$ --- one can check that $\Delta^0\diamond \Delta^1$ has 
two non-degenerate 2-simplices.  

For each fixed simplicial set $K$, the functor $K\diamond (-)\colon \sSet\to (\sSet)_{K/}$ 
preserves colimits (see Section 4.2.1 of \cite{HTT}) and hence admits a right 
adjoint whose value on a map $p\colon K\to X$ in $(\sSet)_{K/}$ is the {\em wide slice}
$X^{p/}$.  We shall prove the following theorem in Section~\ref{sec:stability}. 

\begin{thm} 
\label{thm:B}
Suppose given maps 
\[
A \xrightarrow{u} B\xrightarrow{f} X\xrightarrow{p} Y  
\]
in $\sSet$.  If $p$ is an inner fibration then the induced map 
\[
X^{f/}\to X^{fu/}\times_{Y^{pfu/}} Y^{pf/} 
\]
is a left fibration which is a trivial Kan fibration if $u$ is inner anodyne.  
\end{thm} 

This result is a strengthening of Proposition 4.2.1.6 from \cite{HTT} where 
the result above is proven under the assumption that $p\colon X\to Y$ is a 
categorical fibration.  As far as we are aware Theorem~\ref{thm:B} has not 
appeared in the literature before.  The proof of Theorem~\ref{thm:B} will be easy, once we have 

\begin{thm} 
\label{thm:C}
Suppose $u\colon A\to B$ and $v\colon C\to D$ are monomorphisms of simplicial sets.  
If $v$ is left anodyne then the canonical map 
\[
u\diamond' v\colon B\diamond C \cup A\diamond D\to B\diamond D 
\]
is inner anodyne.  
\end{thm} 

In Section~\ref{sec:stability} we shall also prove the following result, 
which gives an alternative characterization of cocartesian morphisms.  

\begin{thm}
\label{thm:D} 
Let $p\colon X\to S$ be an inner fibration of simplicial sets.  Suppose that $u\colon x\to y$ is an 
edge of $X$.  Then the following statements are equivalent: 
\begin{enumerate}
\item $u$ is a $p$-cocartesian morphism of $X$; 
\item the induced map $X^{u/}\to X^{x/}\times_{S^{p(x)/}}S^{p(u)/}$ is a trivial Kan fibration.   
\end{enumerate}
\end{thm} 

Our main technical tool for the proof of Theorem~\ref{thm:C} will be the following theorem which is 
of interest in its own right.  

\begin{thm} 
\label{thm:E} 
The class of inner anodyne maps has the right cancellation property. 
\end{thm}

Recall (see \cite{JT}) that a class $\cA$ of monomorphisms in $\sSet$ is said to 
satisfy the {\em right cancellation} property if the following is 
true: for all monomorphisms of simplicial sets $i\colon A\to B$ and 
$j\colon B\to C$, if $i, ji\in \cA$ then $j\in \cA$.

In summary then this paper is as follows.  In Section~\ref{sec:inner} we prove 
Theorem~\ref{thm:E} and establish some corollaries of it.  In Section~\ref{sec:stability} we prove Theorem~\ref{thm:C}, 
the stability theorem (Theorem~\ref{thm:B}) and we prove Theorem~\ref{thm:D} 
giving an alternate description of cocartesian morphisms in terms of the wide slice.  
For the most part we will use the notation from \cite{HTT}; thus $\sSet$ will denote 
the category of simplicial sets, for instance.  
% I am very grateful to an anonymous referee 
% for pointing out some errors in the proof of Theorem~\ref{thm:D} in an earlier version of this paper. 

\section{Inner anodyne maps}
\label{sec:inner} 

The aim of this section is to prove Theorem~\ref{thm:E} and explore some corollaries of it.  
To begin with, we give some examples of classes 
of monomorphisms with the right cancellation property.  
The class of anodyne maps in $\sSet$ has the right 
cancellation property (this is a simple consequence of the 2-out-of-3 
property for weak homotopy equivalences).  Less obvious is the 
following result due to Joyal (see Corollary 8.15 of \cite{J2}) and Lurie (see Corollary 4.1.2.2 of \cite{HTT}).  

\begin{proposition}[\cite{J2,HTT}] 
The class of left anodyne maps in $\sSet$ satisfies the 
right cancellation property. 
\end{proposition} 

The right cancellation property forms part of several  
criteria for showing that a class of monomorphisms contains a 
given class as a subclass.  For example 
we have the following result due to Joyal and Tierney (see Lemma 3.5 of 
\cite{JT}) 

\begin{proposition}[\cite{JT}] 
Let $\cA$ be a saturated class of monomorphisms in $\sSet$ 
satisfying the right cancellation property.  If $\cA$ contains 
the inclusions $I_n\subset \Delta^n$ for all $n\geq 2$ 
then $\cA$ contains every inner anodyne map.  
\end{proposition} 

Here $I_n$ denotes the {\em $n$-chain}, i.e.\ 
\[
I_n  = \Delta^{\set{0,1}}\cup \Delta^{\set{1,2}}\cup \cdots \cup \Delta^{\set{n-1,n}}.  
\]
From \cite{S} we have an analogous criteria for left 
anodyne maps which we shall make use of in the following section.  

\begin{proposition}[\cite{S}]
\label{prop:contains left anodynes}
Let $\mathcal{A}$ be a saturated class of monomorphisms in $\sSet$ 
which satisfies the right cancellation property.  Then the following statements 
are equivalent: 
\begin{enumerate} 
\item $\mathcal{A}$ contains the class of left anodyne morphisms; 
\item $\mathcal{A}$ contains the initial vertex maps $\Delta^{\set{0}}\to \Delta^n$ for all $n\geq 1$; 
\item $\mathcal{A}$ contains the horn inclusions $h^0_n\colon \Lambda^n_0\subset \Delta^n$ for all $n\geq 1$.  
\end{enumerate} 
\end{proposition}

We now turn our attention to the proof of Theorem~\ref{thm:E}.  We will need the following lemma.  

\begin{lemma} 
\label{lem:filler}
Let $p\colon X\to S$ be an inner fibration between simplicial sets and suppose given a 
commutative diagram 
\[
\begin{tikzcd}
(\Delta^1\times \partial\Delta^n)\cup (\set{0}\times \Delta^n) \arrow[rr,"u"] \arrow[d] & & X \arrow[d,"p"] \\ 
\Delta^1\times \Delta^n  \arrow[r,"{\mathit{pr}}_1"'] & \Delta^n \arrow[r,"v"'] & S
\end{tikzcd} 
\]
% \[
% \xymatrix{ 
% (\Delta^1\times \partial\Delta^n)\cup (\set{0}\times \Delta^n) \ar[rr]^-u \ar[d] & & X \ar[d]^-p \\ 
% \Delta^1\times \Delta^n  \ar[r]_-{\mathit{pr}_1} & \Delta^n \ar[r]_-v & S} 
% \]
where $n\geq 1$.  If the edge $u|\Delta^1\times \set{0}$ is an equivalence in the $\infty$-category 
$X_{v(0)}$ then 
the square has a diagonal filler.  
\end{lemma} 

Here $X_{v(0)}$ denotes the $\infty$-category which is the fiber of $p\colon X\to S$ over the 
vertex $v(0)$ of $S$.  

\begin{proof} 
By pulling back $p\colon X\to S$ to $\Delta^n$ via $v$ we may suppose without loss of 
generality that $S$ and $X$ are $\infty$-categories.  
We recall some 
facts about the structure of the simplicial set $\Delta^1\times \Delta^n$.  
This simplicial set can be regarded as the nerve of the category $[1]\times [n]$, 
which may be pictured as 
\[
\begin{tikzcd} 
0 \arrow[d] \arrow[r] & 1 \arrow[d] \arrow[r] & \cdots \arrow[r] & n \arrow[d] \\ 
0' \arrow[r] & 1' \arrow[r] & \cdots \arrow[r] & n'. 
\end{tikzcd} 
\]
% \[
% \xymatrix{ 
% 0 \ar[d] \ar[r] & 1 \ar[d] \ar[r] & \cdots \ar[r] & n \ar[d] \\ 
% 0' \ar[r] & 1' \ar[r] & \cdots \ar[r] & n'. } 
% \]
The simplicial set $\Delta^1\times \Delta^n$ is a union of $n+1$ non-degenerate 
$(n+1)$-simplices $\sigma_0$, $\sigma_1$, \ldots, $\sigma_n$ where 
\[
\sigma_0 = \Delta^{\set{0,0',\ldots,n'}},\quad \sigma_1 = \Delta^{\set{0,1,1',\ldots,n'}}, \ldots, 
\sigma_n = \Delta^{\set{0,1,\ldots,n,n'}}.
\]
Following the proof of Proposition 2.1.2.6 from \cite{HTT}, 
we define a chain of inclusions 
\[
A(n+1)\subset A(n)\subset \cdots \subset A(1)\subset A(0) 
\]
of $\Delta^1\times \Delta^n$ where 
\[
A(n+1) = \set{0}\times\Delta^n \cup \Delta^1\times\partial\Delta^n  
\]
and where $A(k)$ denotes the union of $A(k+1)$ and the simplex $\sigma_k$ (together 
with all of its faces).  Then $A(k)$ is the pushout 
\[
A(k+1)\cup_{\Lambda^{n+1}_k}\Delta^{n+1} 
\]
and $A(0) = \Delta^1\times \Delta^n$.  Since the inclusions $A(k+1)\subset A(k)$ are 
inner anodyne for $1\leq k\leq n$ we are reduced to the problem of showing that 
we may extend the map $\Lambda^{n+1}_0\to X_{\sigma_0}$ induced by $u$, over the simplex 
$\sigma_0$ (here $X_{\sigma_0}$ denotes the pullback of $X$ to $\sigma_0$).  
By hypothesis the edge $u|\Delta^1\times \set{0}$ is an equivalence in 
$X_{v(0)}$ and hence in $X_{\sigma_0}$; the result then follows from Proposition 1.2.4.3 of \cite{HTT}.   
\end{proof}

We now give the proof of Theorem~\ref{thm:E}.

\begin{proof}[Proof of Theorem~\ref{thm:E}]  
Suppose that $i\colon A\to B$ and $j\colon B\to C$ are monomorphisms 
of simplicial sets such that $i$ and $ji$ are inner anodyne.  We will 
prove that $j$ is inner anodyne.  It suffices to show that $j$ has 
the left lifting property against all inner fibrations.  

Let $p\colon X\to Y$ be an inner 
fibration between simplicial sets $X$ and $Y$ and suppose given 
a commutative diagram 
\[
\begin{tikzcd}
B \arrow[d,"j"'] \arrow[r,"u"] & X \arrow[d,"p"] \\ 
C \arrow[r,"v"'] & Y. 
\end{tikzcd} 
\]
% \[
% \xymatrix{ 
% B \ar[d]_-j \ar[r]^-u & X \ar[d]^-p \\ 
% C \ar[r]_-v & Y. } 
% \]
We will show that there exists a diagonal filler for this diagram.  
Since $ji$ is inner anodyne we may choose a map $\phi\colon C\to X$ 
such that $p\phi = v$ and $\phi j i = ui$.  
Observe that $\phi j$ and $u$ lie in the same fiber of the canonical 
map $X^B\to Y^B\times_{Y^A}X^A$.  Since this map is a trivial Kan fibration 
(on account of the hypothesis on $i$ --- see Corollary 2.3.2.5 of \cite{HTT}) 
it follows that there exists a homotopy $h\colon \phi j\to u$ relative to 
$A$ over $Y$.   Thus $h$ determines a map $h\colon \Delta^1\times B\to X$ 
such that $h|\set{0}\times B = \phi j$, $h|\set{1}\times B = u$, $ph = vj\pi_B$, 
$h(1\times i) = ui\pi_A$, and $h|\Delta^1\times \set{b}$ is an equivalence 
in $X_{v(b)}$ for all vertices $b$ of $B$.  

Since $ph = vj\pi_B = v\pi_C(1\times j)$, the maps $h$ and $\phi$ determine a map $w\colon \set{0}\times C\cup \Delta^1\times B \to X$ 
forming part of a commutative diagram 
\[
\begin{tikzcd} 
\set{0}\times C\cup \Delta^1\times B \arrow[d] \arrow[rr,"w"] & &  X \arrow[d,"p"] \\ 
\Delta^1\times C \arrow[r,"\pi_C"'] & C \arrow[r,"v"'] & Y. 
\end{tikzcd} 
\]
% \[
% \xymatrix{ 
% \set{0}\times C\cup \Delta^1\times B \ar[d] \ar[rr]^-{w} & &  X \ar[d]^-p \\ 
% \Delta^1\times C \ar[r]_-{\pi_C} & C \ar[r]_-v & Y. } 
% \]
Observe that $w|\Delta^1\times \set{b}$ is an equivalence in the $\infty$-category 
$X_{v(b)}$ for every vertex $b$ of $B$.  
To finish the proof it suffices to find a diagonal filler $d\colon \Delta^1\times C\to X$ 
for this diagram, for then $d|\set{1}\times C\to X$ is a diagonal filler for the original 
diagram above.  

Writing $C = B\cup \bigcup_{n\geq 1} \sk_n C$, where $\sk_n C$ denotes the $n$-skeleton of 
$C$ and noting that $A_0 = B_0 = C_0$, we see that it suffices to prove by induction on 
$n\geq 0$ that there is a diagonal filler for the induced diagrams  
\[
\begin{tikzcd}
\set{0}\times (B\cup \sk_{n+1}C)\cup \Delta^1\times (B\cup \sk_nC)  \arrow[d] \arrow[rr] & &  X \arrow[d,"p"] \\ 
\Delta^1\times(B\cup \sk_{n+1} C)  \arrow[r] & B\cup \sk_{n+1}C \arrow[r,"v"'] & Y. 
\end{tikzcd} 
\]
% \[
% \xymatrix{ 
% \set{0}\times (B\cup \sk_{n+1}C)\cup \Delta^1\times (B\cup \sk_nC)  \ar[d] \ar[rr] & &  X \ar[d]^-p \\ 
% \Delta^1\times(B\cup \sk_{n+1} C)  \ar[r] & B\cup \sk_{n+1}C \ar[r]_-v & Y. } 
% \]
But the map $B\cup \sk_n C\to B\cup\sk_{n+1}C$ forms part of a pushout 
diagram 
\[
\begin{tikzcd}
\bigsqcup \partial\Delta^{n+1} \arrow[d] \arrow[r] & B\cup \sk_n C \arrow[d] \\ 
\bigsqcup \Delta^{n+1} \arrow[r] & B\cup \sk_{n+1}C
\end{tikzcd}
\]
% \[
% \xymatrix{ 
% \bigsqcup \partial\Delta^{n+1} \ar[d] \ar[r] & B\cup \sk_n C \ar[d] \\ 
% \bigsqcup \Delta^{n+1} \ar[r] & B\cup \sk_{n+1}C}
% \]
where the coproduct is over the set of non-degenerate $(n+1)$-simplices of 
$C$ which do not belong to $B$.  Hence we 
may apply Lemma~\ref{lem:filler} to conclude that the required diagonal fillers exist.  
\end{proof} 

We give some corollaries.  First we have 

\begin{lemma} 
\label{lem:pushout of inner anodynes}
Suppose given a commutative diagram 
\[
\begin{tikzcd} 
A_1 \arrow[d,"f_1"'] & \arrow[l,"u"'] \arrow[d,"f_2"'] A_2 \arrow[r] & A_3 \arrow[d,"f_3"] \\ 
B_1 & \arrow[l,"v"] B_2 \arrow[r] & B_3 
\end{tikzcd}
\]
% \[
% \xymatrix{ 
% A_1 \ar[d]_-{f_1} & \ar[l]_-u \ar[d]_-{f_2} A_2 \ar[r] & A_3 \ar[d]_-{f_3} \\ 
% B_1 & \ar[l]^-v B_2 \ar[r] & B_3 }
% \]
of simplicial sets where $f_1$, $f_2$, $f_3$ are 
inner anodyne and $u$, $v$ are monomorphisms.  If the left hand square is a pullback, then the induced map 
\[
A_1\cup_{A_2} A_3 \to B_1\cup_{B_2} B_3 
\]
is inner anodyne.  
\end{lemma} 

\begin{proof} 
The induced map factors as 
\[
A_1\cup_{A_2} A_3 \to A_1\cup_{A_2}B_3 \simeq A_1\cup_{A_2}B_2\cup_{B_2} B_3 \to B_1\cup_{B_2}B_3.   
\]
The first map in this composite is inner anodyne since it is the 
pushout of an inner anodyne map.  The induced map $A_1\cup B_2 = A_1\cup_{A_2}B_2\to B_1$ is 
inner anodyne by Theorem~\ref{thm:E}.  
It follows that the second map in the composite above 
is also the pushout of an inner anodyne map.  
\end{proof} 

\begin{corollary} 
\label{corr:u diamond C inner anodyne}
Suppose that $u\colon A\to B$ is an inner anodyne map in $\sSet$.  Then 
the induced map $u\diamond C\colon A\diamond C\to B\diamond C$ 
is inner anodyne for any simplicial set $C$.  
\end{corollary} 

\begin{proof} 
This follows from two applications of Lemma~\ref{lem:pushout of inner anodynes} and the 
definition of the wide join.  
\end{proof} 

\section{Stability of inner fibrations} 
\label{sec:stability} 

In this section we prove Theorems~\ref{thm:B},~\ref{thm:C} and~\ref{thm:D}.  
We begin with the following observation: if $A\subset C$ and $B\subset D$ are subcomplexes, then $A\diamond B$ is naturally a 
subcomplex of $C\diamond D$ --- this follows from the description of the set of $n$-simplices 
of the wide join given in the Introduction.  

Next, we prove the following proposition.    

\begin{proposition} 
\label{prop:wide join of inner anodynes}
Suppose that $u\colon A\to B$ and $v\colon C\to D$ are monomorphisms of simplicial sets.  
If $u$ is inner anodyne then so is the canonical map 
\[
u\diamond' v\colon A\diamond D\cup B\diamond C \to B\diamond D.  
\]
\end{proposition} 

\begin{proof}
We have a commutative diagram 
\[
\begin{tikzcd} 
A\diamond C \arrow[d] \arrow[r] & A\diamond D \arrow[d] \\ 
B\diamond C \arrow[r] & B\diamond D
\end{tikzcd} 
\]
% \[
% \xymatrix{ 
% A\diamond C \ar[d] \ar[r] & A\diamond D \ar[d] \\ 
% B\diamond C \ar[r] & B\diamond D} 
% \]
in which the vertical maps are inner anodyne by Corollary~\ref{corr:u diamond C inner anodyne}.   
It follows that the canonical map 
\[
A\diamond D \to A\diamond D\cup B\diamond C 
\]
is inner anodyne, since it is the pushout of an inner anodyne map.  Now 
the composite map 
\[
 A\diamond D \to A\diamond D\cup B\diamond C \to B\diamond D 
\]
is inner anodyne, and so is the left hand map.  Therefore the right hand map, i.e.\ the map $u\diamond' v$, 
is inner anodyne by Theorem~\ref{thm:E}.
\end{proof} 

Our aim now is to prove Theorem~\ref{thm:C}.

\begin{proof}[Proof of Theorem~\ref{thm:C}] 
We make a series of reductions to progressively simpler cases.  We show first it 
suffices to prove the proposition in the special case when $A = \emptyset$.  
Assuming the truth of the proposition in this case for the moment, we see that in the commutative diagram 
\[
\begin{tikzcd} 
A\diamond C\cup_C D \arrow[r] \arrow[d] & B\diamond C\cup_C D \arrow[d] \\ 
A\diamond D \arrow[r] & B\diamond D. 
\end{tikzcd}
\]
% \[
% \xymatrix{ 
% A\diamond C\cup_C D \ar[r] \ar[d] & B\diamond C\cup_C D \ar[d] \\ 
% A\diamond D \ar[r] & B\diamond D. }
% \]
the vertical maps are inner anodyne.  The composite map 
\[
B\diamond C\cup_C D \to B\diamond C \cup A\diamond D\to B\diamond D
\]
is inner anodyne, and the left hand map in this composite is inner anodyne since 
it is a pushout of an inner anodyne map.  Therefore, by the right cancellation 
property (Theorem~\ref{thm:E}) of inner anodyne maps it follows that $u\diamond' v$ is also inner anodyne.  

Assume now that $A = \emptyset$.  Since the class of inner anodyne maps is weakly saturated, and the functors $(-)\diamond C\colon 
\sSet\to (\sSet)_{C/}$ and $(-)\diamond D\colon \sSet\to (\sSet)_{D/}$ are compatible 
with colimits (see Section~\ref{sec:intro}), a standard argument using the skeletal filtration of 
$B$ and Lemma~\ref{lem:pushout of inner anodynes} shows that we may reduce to the 
special case in which $u$ is the inclusion $\emptyset\subset  \Delta^n$ for some $n\geq 0$.  If $n\geq 2$ let $I_n\subset \Delta^n$ 
denote the inclusion of the $n$-chain.  In the commutative diagram of monomorphisms  
\[
\begin{tikzcd}
I_n\diamond C\cup_C D \arrow[r] \arrow[d] &  \Delta^n\diamond C\cup_C D\arrow[d] \\ 
I_n\diamond D \arrow[r] &  \Delta^n\diamond D 
\end{tikzcd}
\]
% \[
% \xymatrix{
% I_n\diamond C\cup_C D \ar[r] \ar[d] &  \Delta^n\diamond C\cup_C D\ar[d] \\ 
% I_n\diamond D \ar[r] &  \Delta^n\diamond D }
% \]
the horizontal maps are inner anodyne by Corollary~\ref{corr:u diamond C inner anodyne}. 
Therefore, by the right 
cancellation property of inner anodyne maps (Theorem~\ref{thm:E}) it suffices to prove 
that the left hand vertical map is inner anodyne.  Therefore we are reduced (by Lemma~\ref{lem:pushout of inner anodynes}) to the special 
case in which $u$ is the inclusion $\emptyset\subset   \Delta^0$ or $\emptyset\subset   \Delta^1$.    
In the case where $u$ is the inclusion $\emptyset\subset   \Delta^1$, the map 
\[
\Delta^1\diamond C\cup_C D \to \Delta^1\diamond D 
\]
is isomorphic to the map 
\[
D\cup_{D\times \Delta^1} \left(\left(\Delta^0\diamond C\cup_C D\right) \times \Delta^1\right) \to 
D\cup_{D\times \Delta^1} \left(\left(\Delta^0\diamond D\right)\times \Delta^1\right).
\]
Since the pushout of an inner anodyne map is inner anodyne, we are finally reduced to the special 
case in which $u$ is the inclusion $\emptyset\subset   \Delta^0$.  

Let $\cA$ be the class of monomorphisms $v\colon C\to D$ in $\sSet$ for which 
$u\diamond' v$ is inner anodyne.  To complete the proof of the proposition 
it suffices to show that $\cA$ contains every left anodyne map.  Clearly $\cA$ is saturated, and satisfies the 
right cancellation property by Theorem~\ref{thm:E}.  
Therefore, by Proposition~\ref{prop:contains left anodynes}
it suffices to check that $\cA$ contains the initial vertex maps $\Delta^{\set{0}}\to \Delta^n$ for 
all $n\geq 0$.  If $v$ is the initial vertex map $\Delta^{\set{0}}\to \Delta^n$, then 
$u\diamond' v$ is isomorphic to the canonical inclusion  
\[
i\colon (\Delta^{1}\times \set{0})\cup (\set{1}\times \Delta^n)  \hookrightarrow
\Delta^0\diamond \Delta^n. 
\]
The simplicial set $\Delta^0\diamond \Delta^n = (\Delta^1\times\Delta^n )\cup_{\set{0}\times \Delta^n}\Delta^{0}$ 
is the union of the simplicial subsets $S_0 = \sigma_0$, $S_1$, \ldots, $S_n$ where 
\[
S_1 = \sigma_1\cup_{\Delta^{\set{0,1}}}\Delta^0,\  
S_2 = \sigma_2\cup_{\Delta^{\set{0,1,2}}}\Delta^0,\ldots, 
S_n = \sigma_n\cup_{\Delta^{\set{0,\ldots,n}}}\Delta^0,  
\]
and where the $\sigma_i$ are the non-degenerate $(n+1)$-simplices of $\Delta^1\times \Delta^n$ 
described in Lemma~\ref{lem:filler} above.  Let us write $T_0 = S_0$ and $T_i = T_{i-1}\cup S_i$ for 
$i\geq 1$.  Then we have a sequence of inclusions 
\[
T_0 \subset T_1\subset \cdots \subset T_n = \Delta^0\diamond \Delta^n. 
\]

The map $i$ above factors as 
\[
(\Delta^{1}\times \set{0})\cup (\set{1}\times \Delta^n)  \hookrightarrow T_0 \hookrightarrow 
T_n = \Delta^0\diamond \Delta^n  
%\Delta^{\set{0}}\cup_{\set{0}\times \Delta^n} (\Delta^1\times\Delta^n )
\]
and the inclusion $(\Delta^{1}\times \set{0})\cup (\set{1}\times \Delta^n)  \hookrightarrow T_0$ 
is clearly inner anodyne.  To complete the proof we need to show that $T_0\subset T_n$ 
is inner anodyne.  Clearly it suffices to prove that $T_{i-1}\subset T_i$ is inner 
anodyne for $i=1,\ldots,n$.  For each $i\geq 1$ we have a pushout diagram 
\[
\begin{tikzcd} 
S_i\cap T_{i-1} \arrow[d] \arrow[r] & T_{i-1} \arrow[d] \\ 
S_i \arrow[r] & T_i 
\end{tikzcd} 
\]
% \[
% \xymatrix{ 
% S_i\cap T_{i-1} \ar[d] \ar[r] & T_{i-1} \ar[d] \\ 
% S_i \ar[r] & T_i } 
% \]
with 
\[
S_i\cap T_{i-1} = \Delta^{\set{0,\ldots,i-1,i',\ldots,n'}}\cup_{\Delta^{\set{0,\ldots,i-1}}} \Delta^0.  
\]
It suffices to prove that the map $S_i\cap T_{i-1}\to S_i$ is inner anodyne for each 
$i\geq 1$.  This can be done using Theorem~\ref{thm:E} 
and the fact that the maps 
\[
\Delta^{\set{0,\ldots,i-1}}\cup \Delta^{\set{i-1,i',\ldots,n'}} \to \Delta^{\set{0,\ldots,i-1,i',\ldots,n'}} 
\]
and 
\[
\Delta^{\set{0,\ldots,i}}\cup \Delta^{\set{i,i',\ldots,n'}}\to \Delta^{\set{0,\ldots,i,i',\ldots,n'}} 
\]
are inner anodyne.  

\end{proof} 

We now prove Theorem~\ref{thm:B}.  

\begin{proof}[Proof of Theorem~\ref{thm:B}] 
By adjointness, it suffices to show that $p$ has the right lifting property against 
every map of the form 
\[
u\diamond' v\colon B\diamond C \cup A\diamond D \to B\diamond D 
\]
where $v\colon C\to D$ is left anodyne.  But $u\diamond' v$ is inner  
anodyne by Theorem~\ref{thm:C}.  The case where 
$u$ is inner anodyne follows from Proposition~\ref{prop:wide join of inner anodynes}.  
\end{proof}

The remainder of the paper is devoted to the proof of Theorem~\ref{thm:D}.  
Let us say that an edge $u\colon \Delta^1\to X$ is {\em wide} $p$-cocartesian 
if the induced map $X^{u/}\to X^{x/}\times_{S^{p(x)/}}S^{p(u)/}$ is a trivial 
Kan fibration.  Thus to prove Theorem~\ref{thm:D} we must prove that 
an edge of $X$ is $p$-cocartesian if and only if it is wide $p$-cocartesian.  

Suppose first that $u\colon \Delta^1\to X$ is a wide $p$-cocartesian edge.  We need to 
show that $u$ is $p$-cocartesian.  Therefore we need to show that in any commutative diagram of the form 
\[
\begin{tikzcd} 
\Delta^{\set{0,1}} \arrow[d] \arrow[dr,"u"] &                          \\ 
\Lambda^{n+2}_0 \arrow[r] \arrow[d] & X \arrow[d,"p"] \\ 
\Delta^{n+2} \arrow[r] \arrow[ur,dashed] & S
\end{tikzcd}
\]
% \[
% \xymatrix{ 
% \Delta^{\set{0,1}} \ar[d] \ar[dr]^-u &                          \\ 
% \Lambda^{n+2}_0 \ar[r] \ar[d] & X \ar[d]^-p \\ 
% \Delta^{n+2} \ar[r] \ar@{.>}[ur] & S}
% \]
where $n\geq 0$
the indicated map exists making the diagram commute.  
Since $u$ is wide $p$-cocartesian we may find the indicated diagonal filler 
in the diagram 
\[
\begin{tikzcd} 
\Delta^1 \diamond \partial\Delta^n \cup \set{0}\diamond \Delta^n \arrow[d] \arrow[r] & 
\Lambda^{n+2}_0 \arrow[d] \arrow[r] & X \arrow[d,"p"] \\ 
\Delta^1\diamond \Delta^{n} \arrow[urr,dashed] \arrow[r] & \Delta^{n+2} \arrow[r] & S. 
\end{tikzcd}
\]
% \[
% \xymatrix{ 
% \Delta^1 \diamond \partial\Delta^n \cup \set{0}\diamond \Delta^n \ar[d] \ar[r] & 
% \Lambda^{n+2}_0 \ar[d] \ar[r] & X \ar[d]^-p \\ 
% \Delta^1\diamond \Delta^{n} \ar@{.>}[urr] \ar[r] & \Delta^{n+2} \ar[r] & S. }
% \]
We claim that the map $j\colon \Delta^1 \diamond \partial\Delta^n \cup \set{0}\diamond \Delta^n\to \Lambda^{n+2}_0$ 
is a categorical equivalence.  Proposition 4.2.1.2 of \cite{HTT} implies that 
the maps $\Delta^1 \diamond \partial\Delta^n \to \Delta^1 \star \partial\Delta^n$ 
and $\set{0}\diamond \Delta^n\to \set{0}\star \Delta^n$  
are categorical equivalences.  Therefore, since the Joyal model structure is left proper, it suffices to prove that the canonical map 
\[
\Delta^1 \diamond \partial\Delta^n \cup_{\set{0}\diamond \partial\Delta^n} \set{0}\star \partial\Delta^n 
\to \Delta^1\star \partial\Delta^n 
\]
is a categorical equivalence.  This map forms part of the composite map 
\[
\Delta^1 \diamond \partial\Delta^n \to 
\Delta^1 \diamond \partial\Delta^n \cup_{\set{0}\diamond \partial\Delta^n} \set{0}\star \partial\Delta^n 
\to \Delta^1\star \partial\Delta^n 
\]
and hence it suffices to prove that the first map in this composite  
is a categorical equivalence.  But this map is a pushout of the categorical equivalence $\set{0}\diamond 
 \partial\Delta^n\to \set{0} \star \partial\Delta^n$ along the inclusion 
$\set{0}\diamond \partial\Delta^n\to \Delta^1\diamond \partial\Delta^n$ 
and hence is a categorical equivalence since the Joyal model structure is left proper.  

It follows that we may factor the map $j$ as $j = qi$ where 
$i\colon \Delta^1 \diamond \partial\Delta^n \cup \set{0}\diamond \Delta^n\to A$  
is an ayclic cofibration in the Joyal model structure and $q\colon A\to \Lambda^{n+2}_0$ is a trivial 
Kan fibration.  Let 
\[
\begin{tikzcd}
\Delta^1 \diamond \partial\Delta^n \cup \set{0}\diamond \Delta^n \arrow[r,"i"] \arrow[d] & A \arrow[d] \\ 
\Delta^1\diamond \Delta^n \arrow[r] & B
\end{tikzcd} 
\]
% \[
% \xymatrix{ 
% \Delta^1 \diamond \partial\Delta^n \cup \set{0}\diamond \Delta^n \ar[r]^-i \ar[d] & A \ar[d] \\ 
% \Delta^1\diamond \Delta^n \ar[r] & B} 
% \]
denote the pushout diagram and note that in the 
induced diagram 
\[
\begin{tikzcd} 
A\arrow[r,"q"] \arrow[d] & \Lambda^{n+2}_0 \arrow[d] \\ 
B \arrow[r] & \Delta^{n+2} 
\end{tikzcd} 
\]
% \[
% \xymatrix{ 
% A\ar[r]^-q \ar[d] & \Lambda^{n+2}_0 \ar[d] \\ 
% B \ar[r] & \Delta^{n+2} } 
% \]
the map $B\to \Delta^{n+2}$ is a categorical equivalence by Proposition 4.2.1.2 of 
\cite{HTT} and the fact that the canonical map $\Delta^1\diamond \Delta^n\to B$ in the 
diagram above is an acyclic cofibration.  We have another commutative diagram 
\[
\begin{tikzcd}
A \arrow[d] \arrow[r] & \Lambda^{n+2}_0 \arrow[r] & X \arrow[d,"p"] \\ 
 B \arrow[urr] \arrow[r] & \Delta^{n+2} \arrow[r] & S
\end{tikzcd}
\]
% \[
% \xymatrix{ 
% A \ar[d] \ar[r] & \Lambda^{n+2}_0 \ar[r] & X \ar[d]^- p \\ 
%  B \ar[urr] \ar[r] & \Delta^{n+2} \ar[r] & S}
% \]
where the map $B\to X$ is the canonical map from the pushout.  Factor the map $B\to \Delta^{n+2}$ as $B\to B'\to \Delta^{n+2}$, 
where $B\to B'$ is inner anodyne and $B'\to \Delta^{n+2}$ is an inner fibration (and hence a 
categorical fibration).  It follows that $B'\to \Delta^{n+2}$ is a trivial Kan fibration.  
Since $B\to B'$ is inner anodyne and $p\colon X\to S$ is an inner fibration, we may 
find a map $B'\to X$ forming part of a commutative diagram 
\[
\begin{tikzcd}
A \arrow[d] \arrow[r] & \Lambda^{n+2}_0 \arrow[r]  & X \arrow[d,"p"] \\ 
B'\arrow[urr] \arrow[r] & \Delta^{n+2} \arrow[r] & S 
\end{tikzcd} 
\]
% \[
% \xymatrix{
% A \ar[d] \ar[r] & \Lambda^{n+2}_0 \ar[r]  & X \ar[d]^-p \\ 
% B'\ar[urr] \ar[r] & \Delta^{n+2} \ar[r] & S } 
% \]
To complete the proof that the edge $u$ is $p$-cocartesian it suffices to prove that 
$\Lambda^{n+2}_0\to \Delta^{n+2}$ is a retract of $A\to B'$.  Since $A\to \Lambda^{n+2}_0$ is a 
trivial Kan fibration we may choose a section $\Lambda^{n+2}_0\to A$.  Similarly, since $B'\to 
\Delta^{n+2}$ is a trivial Kan fibration, we may extend the induced map $\Lambda^{n+2}_0 \to B'$ 
along the inclusion $\Lambda^{n+2}_0\to \Delta^{n+2}$ to obtain a section of $B'\to \Delta^{n+2}$, 
which exhibits the desired retraction.  

Suppose now that $u\colon \Delta^1\to X$ is $p$-cocartesian.  We need to show that 
$u$ is wide $p$-cocartesian.  Therefore we need to show that the indicated diagonal filler 
exists in every commutative diagram of the form 
\[
\begin{tikzcd} 
\Delta^1\arrow[d] \arrow[dr,"u"] &                                     \\ 
\Delta^1\diamond \partial \Delta^n\cup \Delta^{\set{0}}\diamond \Delta^n \arrow[d] \arrow[r] & X \arrow[d,"p"] \\ 
\Delta^1\diamond \Delta^n \arrow[r] \arrow[ur,dashed] & S. 
\end{tikzcd} 
\]
% \[
% \xymatrix{ 
% \Delta^1\ar[d] \ar[dr]^-u &                                     \\ 
% \Delta^1\diamond \partial \Delta^n\cup \Delta^{\set{0}}\diamond \Delta^n \ar[d] \ar[r] & X \ar[d]^-p \\ 
% \Delta^1\diamond \Delta^n \ar[r] \ar@{.>}[ur] & S. } 
% \]
% {\color{red} Unpacking} the definition of the wide join, we see that 
% $\Delta^1\diamond \partial \Delta^n\cup \Delta^{\set{0}}\diamond \Delta^n$ is given by the colimit
% \[
% \left( \Delta^1\times \Delta^1 \times \partial\Delta^n\cup_{\partial\Delta^1 \times\Delta^1\times \partial\Delta^n} 
% \Delta^1\sqcup\partial\Delta^n\right)\cup 
% \left(\Delta^1\times \Delta^{\set{0}} \times \Delta^n\cup_{\partial\Delta^1\times\Delta^{\set{0}}\times \Delta^n} 
% \Delta^{\set{0}}\sqcup \Delta^n \right)
% \]
% which in turn can be re-written as
A short calculation shows that $\Delta^1\diamond \partial \Delta^n\cup \Delta^{\set{0}}\diamond \Delta^n$ is isomorphic to  
\[
\Delta^1\times A(n+1)\cup_{\partial\Delta^1\times A(n+1)} \Delta^1\sqcup \Delta^n
\]
using the notation from the proof of Lemma~\ref{lem:filler}.  So the lifting problem above may be re-phrased as 
the problem of finding the indicated diagonal filler in the diagram 
\[
\begin{tikzcd}
\Delta^1 \arrow[d] \arrow[dr,"u"] &                                                                       \\ 
\Delta^1\times A(n+1)\cup_{\partial\Delta^1\times A(n+1)} \Delta^1\sqcup \Delta^n \arrow[r] \arrow[d] & X \arrow[d,"p"] \\ 
\Delta^1 \times A(0)\cup_{\partial\Delta^1\times A(0)} \Delta^1 \sqcup \Delta^n  \arrow[ur,dashed] \arrow[r] & S.      
\end{tikzcd}
\]
% \[
% \xymatrix{ 
% \Delta^1 \ar[d] \ar[dr]^-u &                                                                       \\ 
% \Delta^1\times A(n+1)\cup_{\partial\Delta^1\times A(n+1)} \Delta^1\sqcup \Delta^n \ar[r] \ar[d] & X \ar[d]^-p \\ 
% \Delta^1 \times A(0)\cup_{\partial\Delta^1\times A(0)} \Delta^1 \sqcup \Delta^n  \ar@{.>}[ur] \ar[r] & S.      }
% \]
Since the maps $A(n+1-i)\to A(n-i)$ are inner anodyne for all $i=0,1,\ldots,n-1$ it follows from 
Lemma~\ref{lem:pushout of inner anodynes} that we are reduced to proving that the indicated diagonal filler 
exists in the diagram 
\[
\begin{tikzcd}
\Delta^1 \arrow[d] \arrow[dr,"u"] &                                                                       \\ 
\Delta^1\times A(1)\cup_{\partial\Delta^1\times A(1)} \Delta^1\sqcup \Delta^n \arrow[r] \arrow[d] & X \arrow[d,"p"] \\ 
\Delta^1 \times A(0)\cup_{\partial\Delta^1\times A(0)} \Delta^1 \sqcup \Delta^n  \arrow[ur,dashed] \arrow[r] & S.      
\end{tikzcd}
\]
% \[
% \xymatrix{ 
% \Delta^1 \ar[d] \ar[dr]^-u &                                                                       \\ 
% \Delta^1\times A(1)\cup_{\partial\Delta^1\times A(1)} \Delta^1\sqcup \Delta^n \ar[r] \ar[d] & X \ar[d]^-p \\ 
% \Delta^1 \times A(0)\cup_{\partial\Delta^1\times A(0)} \Delta^1 \sqcup \Delta^n  \ar@{.>}[ur] \ar[r] & S.      }
% \]
Observe that the canonical map $\partial\Delta^1\times A(0)\to \Delta^1\sqcup \Delta^n$ induces a pushout 
diagram 
\[
\begin{tikzcd} 
\Delta^1\times A(1)\cup \partial\Delta^1\times A(0) \arrow[d] \arrow[r] & 
\Delta^1\times A(1)\cup_{\partial\Delta^1\times A(1)}\Delta^1\sqcup \Delta^n \arrow[d] \\ 
\Delta^1\times A(0) \arrow[r] & \Delta^1\times A(0)\cup_{\partial\Delta^1\times A(0)} \Delta^1\sqcup \Delta^n 
\end{tikzcd}
\]
% \[
% \xymatrix{ 
% \Delta^1\times A(1)\cup \partial\Delta^1\times A(0) \ar[d] \ar[r] & 
% \Delta^1\times A(1)\cup_{\partial\Delta^1\times A(1)}\Delta^1\sqcup \Delta^n \ar[d] \\ 
% \Delta^1\times A(0) \ar[r] & \Delta^1\times A(0)\cup_{\partial\Delta^1\times A(0)} \Delta^1\sqcup \Delta^n 
% }
% \]
and that the induced map $\set{0}\times \Delta^1\times \set{0}\to X$ is isomorphic to the map 
$u\colon \Delta^1\to X$ (here $\set{0}\times \Delta^1\times \set{0}$ is thought of as a subcomplex of 
$\partial\Delta^1\times A(0)$ in the obvious way).  
Therefore, it suffices to show that the indicated diagonal filler exists 
in every commutative diagram of the form 
\[
\begin{tikzcd} 
\Delta^1\times A(1)\cup \partial\Delta^1\times A(0) \arrow[r] \arrow[d] & X \arrow[d,"p"] \\ 
\Delta^1 \times  A(0)  \arrow[ur,dashed] \arrow[r] & S.      
\end{tikzcd}
\]
% \[
% \xymatrix{ 
% \Delta^1\times A(1)\cup \partial\Delta^1\times A(0) \ar[r] \ar[d] & X \ar[d]^-p \\ 
% \Delta^1 \times  A(0)  \ar@{.>}[ur] \ar[r] & S.      }
% \]
in which the induced map $\set{0}\times \Delta^{1}\times \set{0}\to X$ represents a 
$p$-cocartesian edge of $X$.  

Therefore, since the map $A(1)\to A(0)$ is obtained as a pushout of $\Lambda^{n+1}_0\to \Delta^{n+1}$, we are finally 
reduced to proving the following lemma.  
% Not suprisingly, this Lemma appears in \cite{HTT} as Lemma 4.2.2.3 (if $A$ 
% is a point in the statement of that lemma then the hypothesis that $q$ is a cartesian fibration can be relaxed 
% to the hypothesis that $q$ is an inner fibration).  We have elected to give a proof of this lemma since some 
% of the details of the proof of the corresponding result in \cite{HTT} are left to the reader.    

\begin{lemma} 
Let $n\geq 1$ and suppose given a commutative diagram  
\[
\begin{tikzcd}
\Delta^1\times \Lambda^{n}_0\cup \partial\Delta^1\times \Delta^{n} \arrow[r,"u"] \arrow[d] & X \arrow[d,"p"] \\ 
\Delta^1 \times  \Delta^{n}  \arrow[ur,dashed,"f"] \arrow[r] & S.      
\end{tikzcd}
\]
% \[
% \xymatrix{ 
% \Delta^1\times \Lambda^{n}_0\cup \partial\Delta^1\times \Delta^{n} \ar[r]^-u \ar[d] & X \ar[d]^-p \\ 
% \Delta^1 \times  \Delta^{n}  \ar@{.>}[ur]^-f \ar[r] & S.      }
% \]
of simplicial sets where $p\colon X\to S$ is an inner fibration.  
If $u|\set{0}\times \Delta^{\set{0,1}}$ is a $p$-cocartesian edge of $X$ then the indicated 
diagonal filler exists.  
\end{lemma}

\begin{proof}
We remark first that 
\[
\Delta^1\times \partial_i\Delta^{n} = \partial_{i+1}\sigma_0\cup \cdots \cup \partial_{i+1}\sigma_{i-1}\cup 
\partial_i \sigma_{i+1}\cup \cdots \cup \partial_i \sigma_n.
\]
where $\partial_i\Delta^n$ denotes the $i$-th face of $\Delta^n$ (i.e.\ $\partial_i\Delta^n = 
\Delta^{\set{0,\ldots,\hat{i},\ldots,n}}$), and where $\partial_i\sigma_j$ denotes 
the $i$-th face of the $(n+1)$- simplex $\sigma_j$.  

Write $B(n+1) = \Delta^1\times \Lambda^{n}_0\cup \partial\Delta^1\times \Delta^{n}$, and assuming 
that $B(i+1)$ has been defined, let $B(i)$ equal the union of $B(i+1)$ and the $(n+1)$-simplex $\sigma_{n-i}$, 
together with all of its faces.  Thus $B(0) = \Delta^1\times \Delta^{n}$ and for each $i=0,1,\ldots,n$ we have a pushout diagram 
\[
\begin{tikzcd}
B(i+1)\cap \sigma_{n-i} \arrow[r] \arrow[d] & B(i+1) \arrow[d] \\ 
 \sigma_{n-i} \arrow[r] & B(i). 
 \end{tikzcd} 
\]
% \[
% \xymatrix{ 
% B(i+1)\cap \sigma_{n-i} \ar[r] \ar[d] & B(i+1) \ar[d] \\ 
%  \sigma_{n-i} \ar[r] & B(i). } 
% \]
We construct the diagonal filler $f$ from the statement of the lemma by descending induction on $n$.  
By the remark above we have an identification 
\[
B(n+1)\cap \sigma_0 = \bigcup_{i\neq 1} \partial_i\sigma_0
\] 
so that the inclusion 
$B(n+1)\cap \sigma_0\subset \sigma_0$ is isomorphic to the inclusion $\Lambda^{n+1}_1\subset \Delta^{n+1}$.  
It follows that the extension $f|B(n)$ exists since $p\colon X\to S$ is an inner fibration.  

Assume that $f|B(i+1)$ has been constructed for $0<i<n$. We have 
\[
B(i+1)\cap \sigma_{n-i} = \bigcup_{j\neq 0,n-i+1}\partial_j \sigma_{n-i}
\]
and 
\[
B(i+1)\cap \partial_{0}\sigma_{n-i} = \bigcup_{j\neq n-i}\partial_j\partial_{0}\sigma_{n-i}. 
\]
Thus the inclusion $B(i+1)\cap \partial_{0}\sigma_{n-i}\subset \partial_{0}\sigma_{n-i}$ is isomorphic to the 
inner anodyne inclusion $\Lambda^{n}_{n-i}\subset \Delta^{n}$.  Since $p$ is an inner fibration, it 
follows that $f|B(i+1)$ extends to a map $f|B(i+1)\cup \partial_{0}\sigma_{n-i}$.  We now have 
\[
(B(i+1)\cup \partial_{0}\sigma_{n-i})\cap \sigma_{n-i} = \bigcup_{j\neq n-i+1}\partial_j \sigma_{n-i}
\]
and it follows that the desired extension $f|B(i)$ exists, using the fact that $p$ is an inner fibration again.  
Finally, we have 
\[
B(1)\cap \sigma_n = \bigcup_{j\neq 0}\partial_j\sigma_n 
\]
and hence the extension $f$ exists because $u|\Delta^{\set{0,1}}$ is $p$-cocartesian.  
\end{proof}

\medskip 

\noindent
{\bf Acknowledgements}: I thank Thomas Nikolaus for some useful conversations.

\end{document}